\documentclass[reqno]{amsart}

\usepackage{times}
\usepackage{mathrsfs}
\usepackage{amsmath}
\usepackage{amsthm}
\usepackage{amsfonts}

\newtheorem{thm}{Theorem}[section]
\newtheorem{lem}[thm]{Lemma}
\newtheorem{cor}[thm]{Corollary}
\newtheorem{prop}[thm]{Proposition}
\newtheorem{rems}[thm]{Remarks}
\newtheorem{rem}[thm]{Remark}

\DeclareMathAlphabet{\mathpzc}{OT1}{pzc}{m}{it}

\numberwithin{equation}{section}

\newcommand{\Wqb}{W_{p,\mathcal{B}}}
\newcommand{\R}{\mathbb{R}}

\newcommand{\N}{\mathbb{N}}
\newcommand{\A}{\mathbb{A}}
\newcommand{\Ac}{\mathcal{A}}
\newcommand{\E}{\mathbb{E}}

\newcommand{\ml}{\mathcal{L}}

\newcommand{\Om}{\Omega}

\newcommand{\ve}{\varepsilon}
\newcommand{\rd}{\mathrm{d}}

\newcommand{\bqn}{\begin{equation}}
\newcommand{\eqn}{\end{equation}}
\newcommand{\bqnn}{\begin{equation*}}
\newcommand{\eqnn}{\end{equation*}}
\newcommand{\bear}{\begin{eqnarray}} 
\newcommand{\eear}{\end{eqnarray}} 
\newcommand{\bean}{\begin{eqnarray*}} 
\newcommand{\eean}{\end{eqnarray*}} 
\newcommand{\bs}{\begin{split}}
\newcommand{\es}{\end{split}}

\newcommand{\dimens}{\mathrm{dim}}
\newcommand{\codim}{\mathrm{codim}}
\newcommand{\kk}{\mathrm{ker}}
\newcommand{\im}{\mathrm{rg}}
\newcommand{\spann}{\mathrm{span}}

\newcommand{\dhr}{\mathrel{\lhook\joinrel\relbar\kern-.8ex\joinrel\lhook\joinrel\rightarrow}}

\title[Equilibrium Solutions for Structured Population Models]
{Positive Equilibrium Solutions for Age and Spatially Structured Population Models}

\author[Ch. Walker]{Christoph Walker}

\address{Leibniz Universit\"at Hannover, Institut f\"ur Angewandte Mathematik, Welfengarten 1, D--30167 Hannover, Germany.}
\email{walker@ifam.uni-hannover.de}

\begin{document}

\begin{abstract}
The existence of positive equilibrium solutions to age-dependent population equations with nonlinear diffusion is studied in an abstract setting. By introducing a bifurcation parameter measuring the intensity of the fertility it is shown that a branch of (positive) equilibria bifurcates from the trivial equilibrium. In some cases the direction of bifurcation is analyzed. 
\end{abstract}

\keywords{Age structure, nonlinear diffusion, population models, bifurcation.
\\
{\it Mathematics Subject Classifications (2000)}: 35K55, 35K90, 92D25.}

\maketitle

\section{Introduction}

\noindent The present paper is dedicated to the study of nontrivial equilibrium (i.e. nonzero time-independent) solutions to abstract age-structured population models with nonlinear diffusion, that is, to equations of the form
\begin{align}
&\partial_t u\,+\, \partial_au \, +\,     A(u,a)\,u\,+\,\mu(u,a)\, u =0\ ,& t>0\ ,\quad a\in (0,a_m)\ ,\label{1}\\ 
&u(t,0)\, =\, \int_0^{a_m}\beta(u,a)\, u(a)\, \rd a\ ,& t>0\ ,\label{2}
\end{align}
subject to some initial condition at $t=0$. Here, $u=u(t,a)$ is a function taking on values in some Banach space $E_0$ and represents in applications the density at time $t$ of a population of individuals structured by age $a\in J:=[0,a_m)$, where $a_m\in (0,\infty]$ is the maximal age. The real-valued functions $\mu=\mu(u,a)$ and $\beta=\beta(u,a)$ are respectively the death and birth modulus. The operator $A(u,a)$ depending in a certain way on the density $u$ specified later governs the spatial movement of individuals. It is assumed to be a (unbounded) linear operator $A(u,a):E_1\subset E_0\rightarrow E_0$ satisfying additional technical assumptions given later.

Age-structured models have a long history and questions regarding well-posedness and behavior for large times were investigated (see \cite{WebbSpringer} and the references therein) though most research was devoted to models neglecting spatial structure from the outset or considering merely linear diffusion, see e.g. \cite{FragnelliManiar,LanglaisJMB,LanglaisBusenberg,MagalThieme,RhandiSchnaubelt,ThiemeDCDS98} and the references therein. Less seems to be known for the case of age-structured models with nonlinear diffusion (however, see e.g. \cite{BusenbergIannelli3,LanglaisSIAM,WalkerEJAM,WalkerDCDS,LauWalJMPA}). 

The understanding of the large time behavior of age-structured populations whose evolution is governed by equations \eqref{1}, \eqref{2} requires in particular precise information about the existence of equilibrium solutions. Since obviously $u\equiv 0$ is such an equilibrium solution the aim is to establish existence of nontrivial equilibria. Moreover, since $u$ represents a density the main task is to single out the positive equilibrium solutions in the (ordered) space $E_0$.

Equilibria of \eqref{1}, \eqref{2}
are solutions to
\begin{align}
&\partial_au \, +\, A(u,a)\,u\,+\,\mu(u,a)\, u =0\ ,\qquad a\in J\ ,\label{4}\\ 
&u(0)\, =\, \int_0^{a_m}\beta(u,a)\, u(a)\, \rd a\ .\label{44}
\end{align}
Supposing that the map (with $u$ fixed)
$$
a\mapsto \mathbb{A}(u,a):=A(u,a)+\mu(u,a)
$$
generates a parabolic evolution operator $\Pi_u(a,\sigma)$, $0\le \sigma\le a<a_m$, on $E_0$, a solution to \eqref{4}, \eqref{44} must satisfy
\bqn\label{5}
u(a)=\Pi_u(a,0)u(0)\ ,\quad a\in J\ ,\qquad u(0)=Q(u)u(0)\ ,
\eqn
that is, $u(0)$ is (if nonzero) an eigenvector corresponding to the eigenvalue 1 of the linear operator $Q(u)$ on $E_0$ (for $u$ fixed) defined by
\bqn\label{a}
Q(u):=\int_0^{a_m}\beta(u,a)\Pi_u(a,0)\,\rd a\ .
\eqn
Roughly speaking, $Q(u)$ contains information about the spatial distribution of the expected number of newborns that an individual produces over its lifetime when the population's distribution is $u$. In the present paper we suggest a bifurcation problem by introducing a bifurcation parameter $n$ which determines the intensity of the fertility without changing its structure. More precisely, we are interested in nontrivial solutions $(n,u)$ (that is, $u\not\equiv 0$) to
\begin{align}
&\partial_au \, +\, A(u,a)\,u\,+\,\mu(u,a)\, u =0\ ,\qquad a\in J\ ,\label{A}\\ 
&u(0)\, =\, n\int_0^{a_m}b(u,a)\, u(a)\, \rd a\ ,\label{AA}
\end{align}
where we put
\bqn\label{3}
n\, b(u,a)\, :=\,\beta(u,a) \ ,
\eqn
with $b$ being a normalized function such that the spectral radius of the bounded linear operator
$$
Q_0:=\int_0^{a_m}b(0,a)\Pi_0(a,0)\,\rd a\ 
$$
equals 1, that is,
\bqn\label{777}
r(Q_0)=1\ .
\eqn
Note that under this normalization we have $r(Q(u))=nr(Q_0)=n$; the bifurcation parameter $n$ is thus the spectral radius of the ``inherent spatial net production rate at low densities'' (technically when $u\equiv0$). If $r(Q_0)$ is an eigenvalue of $Q_0$, then \eqref{777} may be interpreted as that there exists a distribution for which the population is at balance meaning that the birth processes yield exact replacement (provided that death and birth modulus and spatial displacement are described by $\mu(0,\cdot)$, $\beta(0,\cdot)$, and $A(0,\cdot)$).

In this paper we provide a set of $n$ values for which \eqref{A}, \eqref{AA} have nontrivial and positive solutions, respectively, around the critical value $n=1$ and $u\equiv 0$, analogously to the ``spatially homogeneous'' case (i.e. when $A=0$), see \cite{CushingI}. More precisely, it is shown that a branch of nontrivial solutions bifurcates from (i.e. intersects with) the branch of trivial solutions $(n,u)=(n,0)$, $n\in \R$, at the critical value of $n$. In principle, the direction at which bifurcation occurs will be related to (the values at $u\equiv 0$ of the derivatives of) $\mu$, $\beta$, and $A$ computing a parametrization of the branch of nontrivial solutions. In some cases, the direction can be computed explicitly. In particular, examples will be given where supercritical bifurcation occurs.

In order to derive the results we consider problem \eqref{A}, \eqref{AA} in a more general framework so that the results actually apply to a broader range of similar problems. In Section \ref{section 2} we investigate the general abstract framework and prove the bifurcation result using findings based on the implicit function theorem obtained in \cite{Cushing0}. We also derive a more precise characterization of the nontrivial branch of solutions and show that the equilibria are positive. The subsequent Section \ref{section 3} then gives applications for these results. We shall point out that analogue results for populations structured by age only (that is, when $A=0$) were derived \cite{CushingI,CushingII,CushingIII}. Furthermore, additional results regarding equilibrium solutions for age-structured equations are to be found in e.g. \cite{DelgadoEtAl,DelgadoEtAl2,Iannelli,LanglaisJMB,Pruess1,Pruess2,Pruess3,Webb,WebbSpringer} and the references therein.

\section{Abstract Formulation}\label{section 2}

Given Banach spaces $E$ and $F$ we write $\mathcal{L}(E,F)$ for the space of bounded linear operators from $E$ to $F$ equipped with the usual operator norm, and we put $\mathcal{L}(E):=\mathcal{L}(E,E)$. We write $r(A)$ for the spectral radius of $A\in\ml(E)$. The subspace of $\mathcal{L}(E)$ consisting of compact operators is $\mathcal{K}(E)$. If $E$ is an ordered Banach space we write respectively $\mathcal{L}_+(E)$ and $\mathcal{K}_+(E)$ for the corresponding positive operators. We let $\ml is (E,F)$ denote the subspace of $\mathcal{L}(E,F)$ of topological linear isomorphisms. If $E\stackrel{d}{\hookrightarrow} F$, that is, if $E$ is densely embedded in $F$, then $\mathcal{H}(E,F)$ is the set of all negative generators of analytic semigroups on $F$ with domain $E$. $\mathcal{BIP}(E;\phi)$ stands for the set of operators with bounded imaginary powers and power angle $\phi\in [0,\pi/2)$, that is, those linear operators $A$ in $E$ for which there is $M\ge 1$ such that $\| A^{it}\|_{\ml(E)}\le Me^{\phi\vert t\vert}$, $t\in\R$. For details we refer to \cite{LQPP}. \\

Throughout this paper we suppose that $E_0$ is a real Banach space and $E_1\stackrel{d}{\dhr} E_0$, that is, $E_1$ is a densely and compactly embedded subspace of $E_0$. We fix $p\in (1,\infty)$, put $\varsigma:=\varsigma(p):=1-1/p$ and set $E_\varsigma:=(E_0,E_1)_{\varsigma,p}$ with $(\cdot,\cdot)_{\varsigma,p}$ being the real interpolation functor. Similarly we choose for each $\alpha\in (0,1)\setminus\{1-1/p\}$ an arbitrary admissible interpolation functor $(\cdot,\cdot)_\alpha$ and put $E_\alpha:=(E_0,E_1)_\alpha$
so that $E_1\stackrel{d}{\dhr} E_\alpha$ (see \cite{LQPP}). If $E_0$ is ordered by a closed convex cone $E_0^+$, then the interpolation spaces are equipped with the order naturally induced by $E_0^+$. Given $a_m\in (0,\infty]$ we set $J:=[0,a_m)$ which thus may be bounded or unbounded. Moreover, we put
$$\E_0:=L_p(J,E_0)\ ,\qquad \E_1:=L_p(J,E_1)\cap W_p^1(J,E_0)
$$ and recall that
\bqn\label{6}
\E_1\hookrightarrow BUC(J,E_{\varsigma})
\eqn
according to, e.g. \cite[III.Thm.4.10.2]{LQPP}, where $BUC$ stands for bounded and uniformly continuous. In particular, the trace $\gamma u:= u(0)$ is well-defined for $u\in \E_1$. \\

We then study problems of the form
\begin{align}
 \partial_au \, +\,     \A(u,a)\,u\, &=0\ ,\qquad a\in J\ ,\label{7}\\ 
u(0)\,& =\,n\,\ell(u)\ \label{77}
\end{align}
where $\A(u,a)\in \mathcal{L}(E_1,E_0)$ and $\ell(u)\in E_0$ for $a\in J$ and $u\in \E_1$ with $\ell(0)=0$. We will impose more restrictions later. Introducing
$\A_0(a):=\A(0,a)$ and assuming a decomposition
$$
\ell(u)=\ell_0(u)+\ell_*(u)\ 
$$
with a linear part $\ell_0$, we first focus our attention on the linearization around 0 of the above problem.

\subsection{Preliminaries}

In the following we assume that
\bqn\label{9}
\ell_0\in \ml (\E_1,E_\vartheta)\quad\text{for some}\quad \vartheta\in \big(\varsigma,1\big)
\eqn
and that
\bqn\label{10}
\begin{aligned}   & \A_0\in L_{\infty}(J,\ml(E_1,E_0))\ \text{generates a parabolic evolution operator}\\
&\Pi_0(a,\sigma), 0\le\sigma\le a<a_m,\ \text{on}\ E_0\ \text{with regularity subspace}\ E_1\ \text{and}\\
&\text{possesses maximal}\ L_p\text{-regularity, that is,}\ (\partial_a+\A_0,\gamma)\in \mathcal{L}is(\E_1,\E_0\times E_{\varsigma})\ .
\end{aligned}
\eqn
For details about evolution operators and operators possessing maximal regularity we refer the reader, e.g., to \cite{LQPP}. It seems to be worthwhile to point out that, owing to \eqref{10} and \cite[III.Prop.1.3.1]{LQPP}, the problem
$$
\partial_au \, +\,     \A_0(a)\,u\, =\, f(a)\ ,\quad a\in J\ ,\qquad
u(0)\, =\, u^0
$$
admits for each datum $(f,u^0)\in \E_0\times E_{\varsigma}$ a unique solution $u\in \E_1$ given by
\bqn\label{11}
u(a)=\Pi_0(a,0)u^0+\int_0^a \Pi_0(a,\sigma)f(\sigma)\,\rd \sigma\ ,\quad a\in J\ ,
\eqn
satisfying for some $c_0>0$
\bqn\label{12}
\| u\|_{\E_1}\le c_0\big(\|f\|_{\E_0}+\|u^0\|_{E_{\varsigma}}\big)\ .
\eqn
In particular,
\bqn\label{13}
\Pi_0(\cdot,0)\in\ml (E_{\varsigma},\E_1)\quad\text{and}\quad
K_0\in\ml(\E_0,\E_1)\ ,
\eqn
where
$$
(K_0 f)(a):=\int_0^a\Pi_0(a,\sigma)f(\sigma)\,\rd \sigma\ ,\qquad a\in J\ ,\quad f\in \E_0\ ,
$$
while, due to \eqref{9}, 
\bqn\label{15}
\big(f\mapsto \ell_0(K_0 f)\big)\in\ml (\E_0,E_{\varsigma})\ .
\eqn
Moreover, we obtain from \eqref{9}, \eqref{13}, and the fact that $E_\vartheta\dhr E_{\varsigma}$ (e.g. see \cite[I.Thm.2.11.1]{LQPP})
\bqn\label{16}
Q_0\in\ml (E_{\varsigma},E_\vartheta)\cap\mathcal{K}(E_{\varsigma})\quad\text{for}\quad
Q_0w:=\ell_0\big(\Pi_0(\cdot,0)w\big)\ ,\quad w\in E_{\varsigma}\ .
\eqn

\noindent The next result will be fundamental for what follows.

\begin{lem}\label{A1}
Suppose \eqref{9} and \eqref{10}. Then the operator
$$
Lu:=\big(\gamma u-\ell_0(u),(\partial_a+\A_0)u\big)
$$
satisfies $L\in\ml (\E_1,E_\varsigma\times\E_0)$ and has a closed kernel $\kk(L)$ and a closed range $\im(L)$ of finite dimension and codimension, respectively, both of which admit bounded projections $P_k$ and $P_r$. In fact,
$$
\dimens(\kk(L))=\codim(\im(L))=\dimens(\kk(1-Q_0))<\infty
$$
and
\begin{align*}
&\kk(L)=\spann\big\{\Pi_0(\cdot,0)w\,;\ w\in\kk(1-Q_0)\big\}\ ,\\
&\im(L)^\perp:=(1-P_r)\big(E_\varsigma\times\E_0)=\im(1-Q_0)^\perp\times\{0\}\ ,
\end{align*}
where
$E_\varsigma=\im(1-Q_0)\oplus \im(1-Q_0)^\perp$.
\end{lem}

\begin{proof}
First observe that \eqref{11} implies that, for $(h_1,h_2)\in E_\varsigma\times\E_0$, the 
equation
$Lu=(h_1,h_2)$ with $u\in\E_1$
is equivalent to
\bqn\label{17}
u=\Pi_0(\cdot,0)u(0)+K_0h_2\ ,\quad (1-Q_0)u(0)=h_1+\ell_0(K_0 h_2)\ .
\eqn
If $1$ is not an eigenvalue of $Q_0\in\mathcal{K}(E_\varsigma)$, then \eqref{17} easily entails that $\kk(L)$ is trivial. Moreover, in this case we have $h_1+\ell_0(K_0 h_2)\in E_\varsigma$ for any $(h_1,h_2)\in E_\varsigma\times\E_0$ by \eqref{15} and there is a unique $w\in E_\varsigma$ for which
$(1-Q_0)w=h_1+\ell_0(K_0 h_2)$.
Thus $u:=\Pi_0(\cdot,0)w+K_0h_2$ belongs to $\E_1$ due to \eqref{13} and satisfies $Lu=(h_1,h_2)$, whence $\im(L)=E_\varsigma\times\E_0$ from which the claim follows in this case.

Otherwise, if $1$ is an eigenvalue of $Q_0\in\mathcal{K}(E_\varsigma)$, then \eqref{17} ensures
$$
\kk(L)=\spann\big\{\Pi_0(\cdot,0)w\,;\, w\in \kk(1-Q_0)\big\}\subset\E_1
$$
which is clearly closed since $L\in\ml (\E_1,E_\varsigma\times\E_0)$ by \eqref{9}, \eqref{10}. In particular, the dimension of $\kk(L)$ equals the dimension of $\kk(1-Q_0)$, the latter clearly being finite since the eigenvalue 1 has finite multiplicity. Therefore, $\kk(L)$ is complemented in $\E_1$ and admits a bounded projection $P_k\in\ml(\E_1,\kk(L))$ (e.g., see \cite[Lem.4.21]{Rudin}). Next, given $(h_1,h_2)\in \im(L)\subset E_\varsigma\times\E_0$ and $u\in \E_1$ with $Lu=(h_1,h_2)$, we have $h_1+\ell_0(K_0 h_2)\in \im(1-Q_0)$ as observed in \eqref{17}. Conversely, if $(h_1,h_2)\in \E_\varsigma\times\E_0$ and $(1-Q_0)w=h_1+\ell_0(K_0 h_2)$ for some $w\in\E_\varsigma$, then $Lu=(h_1,h_2)$ for $u:=\Pi_0(\cdot,0)w+K_0h_2$. Thus
\bqn\label{18}
\im(L)=\big\{(h_1,h_2)\in E_\varsigma\times\E_0\,;\, h_1+\ell_0(K_0 h_2)\in \im(1-Q_0)\big\}\ .
\eqn
Since $Q_0$ is compact, $M:=\im(1-Q_0)$ is a closed subspace of $E_\varsigma$, hence $\im(L)$ is closed in $E_\varsigma\times\E_0$ due to \eqref{15}. Furthermore, $\codim(M)=\dimens(\kk(1-Q_0))<\infty$, and hence $M$ is complemented in $E_\varsigma$, that is, $E_\varsigma=M\oplus M^\perp$. Let $P_M\in\ml(E_\varsigma)$ denote the projection onto $M$ along $M^\perp$ and set
\bqn\label{18B}
P_r(h_1,h_2):=\big(P_Mh_1-(1-P_M)\ell_0(K_0h_2),h_2\big)\ .
\eqn
Then clearly $P_r^2=P_r\in\ml(E_\varsigma\times \E_0)$ by \eqref{15}, $P_r(E_\varsigma\times\E_0)=\im(L)$ by \eqref{18}, and
$$
(1-P_r)(h_1,h_2)=\big((1-P_M)(h_1+\ell_0(K_0h_2)),0\big)\in M^\perp\times\{0\}\ .
$$
Thus we conclude that $E_\varsigma\times\E_0=\im(L)\oplus \im(L)^\perp$ with $\im(L)^\perp=M^\perp\times\{0\}$, and so $$\codim(\im(L))=\dimens(M^\perp)=\dimens(\kk(1-Q_0))=\dimens(\kk(L))\ .$$
This proves the assertion.
\end{proof}

The verification of \eqref{10} is not a simple task in general. We thus recall conditions that allow us in Section~\ref{section 3} to consider cases for which \eqref{10} is readily verified.

\begin{lem}\label{B1}
Suppose that
\bqn\label{1001}
\A_0\in BUC(J,\ml(E_1,E_0))\ \text{generates a parabolic evolution operator on}\ E_0\ .
\eqn
Further suppose, for each $a\in J$, that $0$ belongs to the resolvent set of $\A(a)$, that
\bqn\label{3001}
 \A_0(a)\ \text{possesses maximal}\ L_p\text{-regularity}\ ,
\eqn
and that
\bqn\label{400a}
\lim_{a\rightarrow\infty}\A_0(a)\ \text{exists in}\ \ml(E_1,E_0)\ \text{if}\ a_m=\infty\ .
\eqn
Then \eqref{10} is satisfied.
\end{lem}

\begin{proof}
This is a consequence of \cite[Thm.1.4]{Saal}.
\end{proof}

\begin{rems}\label{5001}
(a) If $\A_0\in C^\rho(J,\mathcal{H}(E_1,E_0))$ for some $\rho>0$, then it generates a parabolic evolution operator on $E_0$ due to \cite[II.Cor.4.4.2]{LQPP}.\\
(b) In case that $E_0$ is a UMD-space (see \cite{LQPP} for a definition and properties), condition \eqref{3001} holds if for each $a\in J$ there is some angle $\theta(a)\in [0,\pi/2)$ for which $\A_0(a)\in\mathcal{BIP}(E_0;\theta(a))$, see \cite[III.Thm.4.10.7]{LQPP}.
\end{rems}

\subsection{Nonlinear Theory}
We now focus on problem \eqref{7}, \eqref{77}. Let $m\in\N\setminus\{0\}$ and let $\Sigma$ denote an open ball in $\E_1$ centered at $0$ of some positive radius $R_0>0$. Suppose that
\bqn\label{19}
\A\in C^m\big(\Sigma,L_{\infty}(J,\ml(E_1,E_0)\big)\quad\text{and}\quad \A_0:=\A(0)\ \text{satisfies}\ \eqref{10}\ .
\eqn
We set
$\A_*(u):=\A(u)-\A_0
$ and sometimes write $\A(u,a):=\A(u)(a)$ for $u\in\Sigma$, $a\in J$ and accordingly $\A_*(u,a):=\A_*(u)(a)$. We also assume that $\ell$ admits a decomposition 
\bqn\label{19B}
\ell(u)=\ell_0(u)+\ell_*(u)\ ,
\eqn
where the linear part $\ell_0$ satisfies \eqref{9} and $\ell_*$ is such that
$\ell_*(\varepsilon u)=\varepsilon\bar{\ell}_*(\varepsilon, u)$, $u\in \Sigma$, $\vert\ve\vert<1$,
for some function
\bqn\label{20}
\bar{\ell}_*\in C^m((-1,1)\times \Sigma,E_\varsigma)\quad\text{with}\quad \bar{\ell}_*(0,\cdot)=0\ ,\quad D_u\bar{\ell}_*(\ve,\cdot)=0\ .
\eqn
We put
$$
T(\lambda,u):=\lambda\big(\ell_0(u),0\big)+\big((\lambda+1)\ell_*(u),-\A_*(u)u\big)\ ,\quad (\lambda,u)\in\R\times\Sigma\ ,
$$
and note that with $n=\lambda+1$ problem \eqref{7}, \eqref{77} can be 
be re-written as $Lu=T(\lambda,u)$ with $L$ being given in Lemma \ref{A1}. We then introduce $\bar{T}\in C^m(\R\times (-1,1)\times\Sigma, E_\varsigma\times\E_0)$ as
$$
\bar{T}(\lambda,\varepsilon,u):=\lambda\big(\ell_0(u),0\big)+\big((\lambda+1)\bar{\ell}_*(\varepsilon, u),-\A_*(\varepsilon u)u\big)\ ,\quad (\lambda,\varepsilon,u)\in \R\times (-1,1)\times\Sigma\ ,
$$ 
and observe that $T(\lambda,\varepsilon u)=\varepsilon \bar{T}(\lambda,\varepsilon,u)$. Nontrivial solutions to \eqref{7}, \eqref{77} are then provided by the following

\begin{thm}\label{A2}
Suppose \eqref{9}, \eqref{19}, \eqref{19B}, and \eqref{20}. Moreover, suppose that $r(Q_0)=1$ is an eigenvalue of $Q_0\in\mathcal{K}(E_\varsigma)$ with geometric multiplicity $1$, where $Q_0$ is defined in \eqref{16}, and let $B\in E_\varsigma$ be a corresponding eigenvector. Then there exists $\varepsilon_0>0$ such that the problem
\begin{align*}
 \partial_au \, +\,     \A(u,a)\,u\, &=0\ ,\qquad a\in J\ ,\\ 
u(0)\,& =\,n\,\ell(u)\ 
\end{align*}
has a branch of nontrivial solutions $\big\{\big(n(\ve),u(\ve)\big)\in\R^+\times\E_1\,;\, 0<\vert\ve\vert<\ve_0\big\}$ of the form
$$
n(\varepsilon)=1+\lambda(\varepsilon)\ ,\qquad u(\varepsilon)=\varepsilon\big(\Pi_0(\cdot,0)B+z(\varepsilon)\big)\ ,\quad 0<\vert \varepsilon\vert<\varepsilon_0\ ,
$$
where $\lambda:(-\varepsilon_0,\varepsilon_0)\rightarrow \R$ and $z:(-\varepsilon_0,\varepsilon_0)\rightarrow \kk(L)^\perp$ are $m$-times continuously differentiable with $\lambda(0)=0$ and $z(0)=0$.
\end{thm}

\begin{proof}
We re-write \eqref{7}, \eqref{77} as $Lu=T(\lambda,u)$ and validate the requirements for Theorem 1 in \cite{Cushing0}. First recall that Lemma \ref{A1} warrants that $L\in\ml(\E_1,E_\varsigma\times\E_0)$ has a closed kernel $\kk(L)=\spann\big\{\Pi_0(\cdot,0)B\big\}$ and a closed range $\im(L)$ both admitting bounded projections $P_k$ and $P_r$, respectively, and that the codimension of $\im(L)$ equals 1. Thus H1 and H2 in \cite{Cushing0} hold. To validate H3 therein we just have to observe that for $y\in \kk(L)\cap\Sigma$
$$
\bar{T}(0,0,y)=\big(\bar{\ell}_*(0,y), -\A_*(0)y\big)=(0,0)
$$
and
$$D_3 \bar{T}(0,0,y)=\big(0,-\A_*(0)y\big)=(0,0)\ .
$$
It remains to verify H4 in \cite{Cushing0}. For, let $1-P_r$ be the projection of $E_\varsigma\times \E_0$ onto the one-dimensional space $\im(L)^\perp=M^\perp\times\{0\}$ with $M=\im(1-Q_0)$ and let $c(\lambda,\varepsilon,z)$ be the component of $\bar{T}(\lambda,\varepsilon,\Pi_0(\cdot,0)B+z)$ with respect to the basis $\big\{(B,0)\big\}$ of $\im(L)^\perp$ for given $\lambda\in\R$, $\vert\varepsilon\vert<1$, and $\|z\|_{\E_1}<R_0/2$. Here we may assume without loss of generality that $\|\Pi_0(\cdot,0)B\|_{\E_1}<R_0/2$. Hence it follows from $Q_0B=B\in M^\perp$, \eqref{20}, and \eqref{18B} that
$$
(1-P_r)\bar{T}(\lambda,0,\Pi_0(\cdot,0)B)=(1-P_r)(\lambda B,0)=\lambda(B,0)\ ,
$$
that is, $c_\lambda(\lambda,0,0)=1$. Now \cite[Thm.1]{Cushing0} implies the assertion.
\end{proof}

\begin{rem}
Clearly, the result applies to non-homogeneous problems
\begin{align*}
 \partial_au \, +\,     \A(u,a)\,u\, &=g(u,a)\ ,\qquad a\in J\ ,\\ 
u(0)\,& =\,n\,\ell(u)\ 
\end{align*}
as well provided that there is $\bar{g}$ such that $g(\ve u,\cdot)=\ve \bar{g}(\ve, u,\cdot)$ with $$[(\ve,u)\mapsto \bar{g}(\ve,u,\cdot)]\in C^m((-1,1)\times\Sigma,\E_0)\quad\text{and}\quad \bar{g}(0,\cdot,\cdot)=0\ .
$$
\end{rem}

Next we compute the $\varepsilon$-expansion of the branch $(n(\varepsilon),u(\varepsilon))$. Under the assumptions of Theorem \ref{A2} we write $\E_1=\kk(L)\oplus \kk(L)^\perp$ and let $P_k\in\ml (\E_1)$ denote the projection onto $\kk(L)=\spann\big\{\Pi_0(\cdot,0)B\big\}$ such that $P_k u=\mathpzc{k} (u)\Pi_0(\cdot,0)B$ with $\mathpzc{k} (u)\in\R$
for $u\in \E_1$. Again we set $M=\im(1-Q_0)$ and $E_\varsigma=M\oplus M^\perp$ with corresponding projection $P_M$ (see the proof of Lemma \ref{A1}).

\begin{prop}\label{A3}
In addition to the assumptions of Theorem \ref{A2} with $m\ge 2$ suppose that $D_u\ell_*(0)=0$. Then the branch of nontrivial solutions $(n(\varepsilon),u(\varepsilon))$, $\vert \varepsilon\vert<\varepsilon_0$, from Theorem \ref{A2} can be written in the form
$$
n(\varepsilon)=1+\zeta\varepsilon+n_*(\varepsilon)\ ,\qquad u(\varepsilon)=\varepsilon\Pi_0(\cdot,0)B+\varepsilon^2\big(\Pi_0(\cdot,0)\xi-K_0h\big)+\varepsilon u_*(\varepsilon)
$$
for $\vert \varepsilon\vert<\varepsilon_0$, where $n_*: (-\varepsilon_0,\varepsilon_0)\rightarrow\R$ and $u_*:(-\varepsilon_0,\varepsilon_0)\rightarrow \kk(L)^\perp$ are such that $\vert n_*(\varepsilon)\vert=o(\varepsilon^2)$ and $\|u_*(\varepsilon)\|_{\E_1}= o(\varepsilon^2)$ as $\vert \varepsilon\vert\rightarrow 0$. The function $h\in\E_0$ is defined by
$$
h(a):=\big(D_u\A_*(0)(\Pi_0(\cdot,0)B)(a)\big)\Pi_0(a,0)B\ ,\quad a\in J\ ,
$$
$\zeta\in\R$ is the unique coefficient of
$$
(1-P_M)\big(\ell_0(K_0h)-g\big)=\zeta B\in M^\perp
$$
with 
$$
g:=\frac{1}{2}D_u^2\ell_*(0)[\Pi_0(\cdot,0)B,\Pi_0(\cdot,0)B]\in E_\varsigma\ ,
$$
and $\xi\in E_\varsigma$ is the unique solution to
$$
(1-Q_0)\xi=\zeta B+g-\ell_0(K_0h)\in M\ ,\quad \mathpzc{k} (\Pi_0(\cdot,0)\xi)=\mathpzc{k} (K_0h)\ .
$$

\end{prop}

\begin{proof}
We plug the twice continuously differentiable functions $\lambda=\lambda(\ve)$ and $u=u(\ve)$ provided by Theorem \ref{A2} into the equation $Lu=T(\lambda,u)$ which we then differentiate twice with respect to $\varepsilon$. Evaluating the result at $\varepsilon=0$ and using $D_u\ell_*(0)=0$ together with $\ell_0(\Pi_0(\cdot,0)B)=B$, we obtain
\bqn\label{21}
L z'(0)=\big(\lambda'(0)B+g,-h\big)
\eqn
with dashes denoting derivatives with respect to $\varepsilon$ and $g,h$ as given in the statement. Hence, from \eqref{17},
$$
y:=\lambda'(0)B+g-\ell_0(K_0h)\in M
$$
and thus, since $P_MB=0$,
$$
(1-P_M)\big(-g+\ell_0(K_0h)\big)=\lambda'(0)B
$$
from which the formula for $n(\varepsilon)$ follows by setting $\zeta:=\lambda'(0)$. Next, if $\varrho\in E_\varsigma$ is an arbitrarily fixed solution to $(1-Q_0)\varrho=y$, then any other $\eta\in E_\varsigma$ with $(1-Q_0)\eta=y$ can be written uniquely in the form $\eta_\alpha:=\eta=\varrho+\alpha B$ for some $\alpha\in\R$. Writing $w:=z'(0)\in\E_1$ we have $w=\Pi_0(\cdot,0)\eta_\alpha-K_0 h$ by \eqref{21} and \eqref{17} with $\alpha\in\R$ determined by the constraint that $w$ must belong to $\kk(L)^\perp$. This is obtained by observing that
$$
0=P_kw=\big(\mathpzc{k}(\Pi_0(\cdot,0)\varrho)+\alpha-\mathpzc{k}(K_0h)\big)\Pi_0(\cdot,0)B\ ,
$$
that is, $\alpha=\mathpzc{k}(K_0h)-\mathpzc{k}(\Pi_0(\cdot,0)\varrho)$. For this $\alpha$ we put $\xi:=\eta_\alpha$ and get 
$$
u(\varepsilon)=\varepsilon\Pi_0(\cdot,0)B+\varepsilon^2\big(\Pi_0(\cdot,0)\xi-K_0h\big)+\varepsilon u_*(\varepsilon)
$$
with $\|u_*(\varepsilon)\|_{\E_1}= o(\varepsilon^2)$ as $\vert\varepsilon\vert\rightarrow 0$.

\end{proof}

\subsection{Positive Solutions}

We shall give conditions under which the nontrivial equilibrium solutions are positive. To this end we suppose that
\bqn\label{22}
E_0\ \text{is ordered by a closed convex cone}\ E_0^+\ .
\eqn
Then the interpolation spaces $E_\sigma$ are given their natural order induced by the cone $E_\sigma^+:=E_\sigma\cap E_0^+$. For information on positive and strongly positive operators we refer to \cite{DanersKochMedina,Schaefer}.
If $(n,u)$ is a solution to \eqref{7}, \eqref{77} we say that $u$ is a positive equilibrium provided that $u(a)\in E_0^+$ for $a\in J$.\\

Before turning to positive solutions we remark the following about the assumptions on $Q_0$ in Theorem~\ref{A2}.

\begin{rem}
Assume that the parabolic evolution operator $\Pi_0(a,\sigma)$ corresponding to $\A_0$ in \eqref{10} is positive, that is, $\Pi_0(a,\sigma)\in\ml_+(E_0)$ for $0\le  \sigma\le a<a_m$. If also $\ell_0\in\ml_+(\E_1,E_\vartheta)$ in \eqref{9}, then $Q_0\in\mathcal{K}_+(E_\varsigma)$ and thus the Krein-Rutman theorem entails that the spectral radius $r(Q_0)$ is (if nonzero) an eigenvalue of finite multiplicity with a positive eigenvector $B\in E_\varsigma^+$. Hence, in this case the assumption in Theorem \ref{A2} that the normalized spectral radius $r(Q_0)=1$ is an eigenvalue is no severe restriction. More restrictive is the assumption that this eigenvalue has geometric multiplicity 1. However, if $Q_0$ is strongly positive or irreducible, then $r(Q_0)=1$ has geometric multiplicity 1, see for instance \cite[Sect.12]{DanersKochMedina} and \cite[App.3.2]{Schaefer}. We also refer to the next section for concrete examples.
\end{rem}

\begin{prop}\label{A4}
Suppose the assumptions of Theorem \ref{A2} and \eqref{22}. In addition,
\bqn\label{u}
\begin{aligned}
&\text{for each}\ u\in\Sigma \ \text{let}\ \A(u)\ \text{generate a positive parabolic}\\
& \text{evolution operator}\ \Pi_u(a,\sigma)\,,\, 0\le\sigma\le a<a_m\, ,\ \text{on}\ E_0\ .
\end{aligned}
\eqn
If $(n(\ve),u(\ve))$ is the branch of solutions from Theorem \ref{A2}, then $u(\varepsilon)$ is positive provided that $\varepsilon\in (0,\varepsilon_0)$ is such that
\bqn\label{23}
\frac{1}{\varepsilon}\gamma u(\varepsilon)=B+\gamma z(\varepsilon)\in E_\varsigma^+\ .
\eqn
In particular, if $B$ belongs to the interior of $E_\varsigma^+$, then $u(\ve)$ is positive for $\varepsilon>0$ sufficiently small.
\end{prop}

\begin{proof}
This follows from the fact that under the stated assumptions any solution $(n,u)$ to \eqref{7}, \eqref{77} satisfies
$$
u(a)=\Pi_u(a,0)u(0)\ ,\quad a\in J\ ,
$$
hence $u(a)\in E_\varsigma^+$ for $a\in J$ if $\gamma u=u(0)\in E_\varsigma^+$, and due to the fact that $z(\varepsilon)\rightarrow 0$ in $\E_1\hookrightarrow BUC(J,E_\varsigma)$ as $\varepsilon\rightarrow 0$.
\end{proof}

\begin{rem}\label{B2}
Recall that according to \cite[II.Cor.4.4.2, II.Thm.6.4.2]{LQPP}, $\A(u)$ generates a positive parabolic evolution operator $\Pi_u(a,\sigma)$ on $E_0$
provided that $\A(u)\in C^\rho(J,\mathcal{H}(E_1,E_0))$
for some $\rho>0$ and $-\A(u)(a)$ is resolvent positive for each $a\in J$. In this case, a solution $u\in\E_1$ to \eqref{7}, \eqref{77} possesses additional regularity, see \cite[II.Thm.1.2.1, II.Thm.5.3.1]{LQPP}.
\end{rem}

Under some symmetry conditions on $\A$ and $\ell$ the equilibrium solutions provided by Theorem \ref{A2} are positive for each parameter value $n(\ve)$, $-\ve_0<\ve<\ve_0$. More precisely, we have:

\begin{prop}\label{B5}
Suppose the assumptions of Theorem \ref{A2}, \eqref{22}, and \eqref{u}. Let $\A(u)=\A(-u)$ and $\ell(u)=-\ell(-u)$ for $u\in\Sigma$. Given $u\in\Sigma$ set $Q_uw:=\ell(\Pi_u(\cdot,0)w)$, $w\in E_\alpha$, and suppose that $Q_u\in\ml_+(E_\alpha)$ for some $\alpha\in[0,\varsigma]$. Moreover, suppose that any positive eigenvalue of $Q_u$ has geometric multiplicity 1 and possesses a positive eigenvector.
Then
$$C^+:=\big\{\big(n(\ve),u(\ve)\big)\,;\, \gamma u(\ve)\in E_0^+\big\}\cup\big\{\big(n(\ve),-u(\ve)\big)\,;\, \gamma u(\ve)\not\in E_0^+\big\}
$$
consists of positive equilibria only.
\end{prop}

\begin{proof} Let $\ve\in (-\ve_0,\ve_0)\setminus\{ 0\}$. Since $(n(\ve),u(\ve))$ satisfies
$$
u(\ve)=\Pi_{u(\ve)}(\cdot,0)\gamma u(\ve)\ ,\quad \gamma u(\ve)=n(\ve) Q_{u(\ve)}\gamma u(\ve)\ ,
$$
it follows that $n(\ve)^{-1}>0$ is an eigenvalue of $Q_{u(\ve)}$ with eigenvector $\gamma u(\ve)$. By assumption there is a corresponding positive eigenvector $B_{u(\ve)}$ and $\alpha_\ve\in\R\setminus\{0\}$ such that $\gamma u(\ve)=\alpha_\ve B_{u(\ve)}$. If $\alpha_\ve>0$ then $\gamma u(\ve)\in E_0^+$ and thus $u(\ve)(a)\in E_0^+$ for each $a\in J$. Otherwise, if $\alpha_\ve<0$, then $-u(\ve)$ is a positive equilibrium solution with parameter value $n(\ve)$ due to $\gamma(-u(\ve))=-\alpha_\ve B_{u(\ve)}\in E_0^+$ and owing to the symmetry conditions put on $\A$ and $\ell$.
\end{proof}

Proposition \ref{A4} guarantees that a branch of positive equilibria bifurcates from the branch of trivial equilibria $(n,u)=(n,0)$, $n\in\R$, at the critical value $n=1$. Near the critical value $n=1$ the set of $n$ values corresponding to positive equilibria on the branch from Theorem \ref{A2} consists of $n$ values greater (i.e. supercritical bifurcation) or less (i.e. subcritical bifurcation) than 1 depending on the sign of $\lambda(\varepsilon)=n(\varepsilon)-1$ for $\varepsilon>0$ sufficiently small. If $m\ge 2$ in Theorem \ref{A2}, this ``direction of bifurcation'', that is, the cases $n(\varepsilon)>1$ and $n(\varepsilon)<1$ for $\varepsilon>0$ small, depends on the sign of $\lambda'(0)=\zeta$ (if nonzero), which in turn depends on $\big(D_u(A_*(0)\big)\Pi_0(\cdot,0)B$ and $D_u^2\ell_*(0)[\Pi_0(\cdot,0)B,\Pi_0(\cdot,0)B]$ according to Proposition \ref{A3}. Further, Proposition \ref{B5} warrants under the symmetry conditions imposed that for any of the values $n(\ve)\ne 1$ there is a positive nontrivial equilibrium. Examples to which Propositions \ref{A4} and \ref{B5} apply will be given in the next section. \\

Under additional assumptions we can get more information about the positive equilibria and the direction of bifurcation. For simplicity we demonstrate this when $\ell$ is given by
\bqn\label{24}
\ell(u):=\int_0^{a_m}b(u,a) u(a)\,\rd a\ ,\quad u\in\Sigma\ ,
\eqn
where $b\in C^m(\Sigma,L_{p'}^+(J))$ with $1/p+1/p'=1$ and $b(u,a):=b(u)(a)$. Then \eqref{19B} and \eqref{20} clearly hold by putting
\bqn\label{24B}
\ell_0(u):=\int_0^{a_m}b_0(a) u(a)\,\rd a\ ,\quad \ell_*(u):=\int_0^{a_m}b_*(u,a) u(a)\,\rd a\ 
\eqn
for $b_0(a):=b(0,a)$ and $b_*(u,a):=b(u,a)-b_0(a)$. Let the assumptions of Proposition \ref{A4} be satisfied and suppose that there exists $\ve_*\in (0,\ve_0)$ such that \eqref{23} holds for $\ve\in (0,\ve_*)$. Let \eqref{u} hold and, given $u\in\Sigma$, assume that
$$
Q_u:=\int_0^{a_m}b(u,a) \Pi_u(a,0)\,\rd a\ ,
$$
belongs to $\mathcal{K}_+(E_\varsigma)$. Note that $Q_u$ for $u=0$ coincides with $Q_0$ defined in \eqref{16}. Set $$N_i:=\inf_{u\in\Gamma} r(Q_u)\ ,\quad N_s:=\sup_{u\in\Gamma} r(Q_u)\ ,$$
where $\Gamma:=\big\{u(\ve)\,;\, \ve\in [0,\ve_*)\big\}$. Then $0\le N_i\le 1\le N_s\le\infty$ since $r(Q_0)=1$. Moreover,
\bqn\label{26}
n\, r(Q_u)\ge 1\quad\text{for}\quad (n,u)\in \Lambda:=\big\{\big(n(\ve),u(\ve)\big)\,;\, \ve\in [0,\ve_*)\big\}\ .
\eqn
Indeed, given $(n,u)\in\Lambda\setminus\{(1,0)\}$ we have $u(a)=\Pi_u(a,0)u(0)$ for $a\in J$ and $$0\not= u(0)=n\ell(u)=n Q_u u(0)\ ,$$
that is, $1/n$ is an eigenvalue of $Q_u\in\ml (E_\varsigma)$, whence $r(Q_u)\ge 1/n$.
Suppose in addition that
\bqn\label{266}
\text{for each}\ u\in\Sigma\,,\,
r(Q_u)>0\ \text{is the only eigenvalue of}\ Q_u\in \mathcal{K}_+(E_\varsigma)\ \text{with positive eigenvector}\ .
\eqn
This holds, e.g., if $Q_u$ is strongly positive. Then 
\bqn\label{27}
n\, r(Q_u)=1\ ,\quad (n,u)\in \Lambda\ .
\eqn
Furthermore, letting
$$
[\sigma_i,\sigma_s]:=cl_\R\big\{n(\ve)\,;\, \ve\in [0,\ve_*)\big\}\ ,$$
it readily follows from \eqref{27} that
\bqn\label{28}
0\le \sigma_i=\frac{1}{N_s}\le 1\le \sigma_s=\frac{1}{N_i}\le \infty\ .
\eqn
Therefore, under the assumptions of Proposition \ref{A4}, \eqref{24}, \eqref{266}, and if $r(Q_u)\le 1$ for $u\in\Sigma$, we have $N_s\le 1$, hence $1=N_s=\sigma_i$ and bifurcation must be supercritical in this case. Again, we refer to the next section for concrete examples.

\section{Applications to Population Dynamics}\label{section 3}

We now apply the obtained results to problem \eqref{7}, \eqref{77}. Suppose \eqref{22} and that the interior $\textrm{int}(E_\varsigma^+)$ of $E_\varsigma^+$ is nonempty. Let $\A$ be of the form
$$
\A(u,a):=\mu(u,a)+A(u,a)
$$
and
$$
\ell(u):=\int_0^{a_m} b(u,a) u(a)\, \rd a\ .
$$
As observed in the previous section, $\ell$ satisfies \eqref{19B}, \eqref{20} with \eqref{24B} provided that
\bqn\label{36}
b\in C^m\big(\Sigma,L_{p'}^+(J)\big)\ ,\quad b_0:=b(0)\not\equiv 0\ ,
\eqn
for some $m\ge 1$ and some ball $\Sigma$ in $\E_1$ centered at 0 with radius $R_0>0$. Moreover, regarding Proposition~\ref{A3} we note that $D_u\ell_*(0)=0$. Let $\alpha\in [0,\varsigma)$ and let $\Phi$ be the ball in $E_\alpha$ with center 0 and radius $R>0$. Let
\bqn\label{30}
A\in C^m(\Phi,\ml(E_1,E_0))
\eqn
be such that 
\bqn\label{35}
-A(w)\ \text{generates an analytic semigroup on}\ E_0 \ \text{and is resolvent positive for each}\ w\in \Phi\ .
\eqn
Making $R_0>0$ smaller if necessary it follows from the compact embedding $E_{\varsigma}\dhr E_\alpha$ and \eqref{6} analogously to \cite[Thm.6.2,Thm.6.4]{AmannEscherII} that the Nemitskii operator of $A$ (again labeled $A$), given by
$$
A(u)(a):=A(u(a))\ ,\quad a\in J\ ,\quad u\in \Sigma\ ,
$$
belongs to $C^m(\Sigma,L_\infty(J,\ml(E_1,E_0)))$. Since \mbox{$\E_1\hookrightarrow BUC^{\varsigma-\delta}(J,E_\delta)$} for $\delta\in[0,\varsigma)$ owing to \eqref{6} and the interpolation inequality \cite[I.Thm.2.11.1]{LQPP}, we deduce from \eqref{35} and Remark \ref{B2} that
$$
[a\mapsto A(u(a))]\in C^{\varsigma-\alpha}(J,\mathcal{H}(E_1,E_0))
$$
generates a positive parabolic evolution operator $U_{A(u)}(a,\sigma)$ on $E_0$ for each $u\in \Sigma$. 
Set $A_0:=A(0)$ and suppose there exist $\omega_0\ge 0$ and $\phi\in [0,\pi/2)$ such that $\omega_0>\mathrm{type}(-A_0)$ and
\bqn\label{31}
\omega_0+A_0\in \mathcal{BIP}(E_0;\phi)\ .
\eqn
Moreover, suppose that
\bqn\label{t}
e^{-a A_0}\in\ml(E_\varsigma)\ \text{is strongly positive for}\ a>0\ .
\eqn
If $\mu$ is a function such that
\bqn\label{32}
[u\mapsto\mu(u,\cdot)]\in C^m(\Sigma,L_\infty^+(J))\ ,
\eqn
we set $\mu_0(a):=\mu(0,a)$ for $a\in J$ and further suppose that
\bqn\label{33}
\mu_0\in BUC(J)\ ,\quad \inf_{a\in J}\,\mu_0(a)>\omega_0\ 
\eqn
and
\bqn\label{34}
\lim_{a\rightarrow\infty}\,\mu_0(a)\ \text{exists if}\ a_m=\infty\ .
\eqn
Put
$
\A(u,a):=\mu(u,a)+A(u(a))$ for $a\in J$, $u\in \Sigma$
and note that
$$
\A_0(a):=\A(0,a)=\mu_0(a)+A_0\ ,\quad a\in J\ .
$$
Clearly, $\A(u,\cdot)$ generates a positive parabolic evolution operator $\Pi_{u}(a,\sigma)$ on $E_0$ for each $u\in \Sigma$ given by
$$
\Pi_u(a,\sigma):=e^{-\int_\sigma^a \mu(u,r)\rd r}\, U_{A(u)}(a,\sigma)\ ,\quad 0\le\sigma\le a<a_m\ .
$$
From \eqref{31}, \eqref{33}, and \cite[III.Cor.4.8.6]{LQPP} it follows that we may apply Remark \ref{5001}.b) to conclude that \eqref{3001} holds true. Proposition \ref{B1} now guarantees that $\A$ satisfies \eqref{10} provided $E_0$ is an UMD space.

Finally, let $Q_0\in\mathcal{K}_+(E_{\varsigma})$ be given by
$$
Q_0:=\int_0^{a_m}b_0(a)\,e^{-\int_0^a \mu_0(r)\rd r}\, e^{-a A_0}\,\rd a
$$
and note that \eqref{t} and $b_0\not\equiv 0$ imply that $Q_0\in\mathcal{K}(E_\varsigma)$ is strongly positive, hence irreducible (see \cite[Sect.12]{DanersKochMedina}). In particular, since the interior of $E_\varsigma^+$ is assumed to be nonempty, it follows from \cite[Thm.12.3]{DanersKochMedina} that $r(Q_0)>0$ is a simple eigenvalue of $Q_0$ with a corresponding eigenvector $B\in \textrm{int}(E_\varsigma^+)$ (and this is the only eigenvalue with a positive eigenvector). Let then $b_0$ be normalized such that $r(Q_0)=1$.\\

Combining Proposition \ref{B1}, Remark \ref{5001}, Proposition \ref{A4}, and Theorem \ref{A2} we obtain:

\begin{thm}\label{B10}
Let $E_0$ be a UMD space satisfying \eqref{22} and let $\textrm{int}(E_\varsigma^+)\ne \emptyset$. Suppose \eqref{36}-\eqref{34}. Then the problem
\begin{align*}
&\partial_a u+A(u(a))u+\mu(u,a)u=0\ ,\quad a\in J\ ,\\
&u(0)=n\int_0^{a_m} b(u,a)u(a)\,\rd a\ ,
\end{align*}
has a branch of nontrivial solutions $\big(n(\ve),u(\ve)\big)\in\R^+\times\E_1$, $0<\vert\ve\vert<\ve_0$, of the form
$$
n(\ve)=1+\lambda(\ve)\ ,\quad u(\ve)=\ve\big(\Pi_0(\cdot,0)B+z(\ve)\big)
$$
such that $\lambda:(-\ve_0,\ve_0)\rightarrow\R$, $z:(-\ve_0,\ve_0)\rightarrow \E_1$ are $m$-times continuously differentiable and $\lambda(0)=0$, $z(0)=0$. If $\ve>0$ is sufficiently small, then $u(\ve)(a)\in E_{\varsigma}^+$ for $a\in J$.
\end{thm}

If, in addition, the symmetry conditions
\bqn\label{40}
A(-u)=A(u)\ ,\quad \mu(-u,\cdot)=\mu(u,\cdot)\ ,\quad b(-u,\cdot)=b(u,\cdot)
\eqn
hold for $u\in \Sigma$ and if
$$
Q_u:=\int_0^{a_m}b(u,a)\,e^{-\int_0^a\mu(u,r)\rd r}\, U_{A(u)}(a,0)\,\rd a
$$
for $u\in\Sigma$ is such that $Q_u\in\ml_+(E_{\varsigma})$ and
\bqn\label{41}
\begin{aligned}
&\text{any positive eigenvalue of}\ Q_u\ \text{has geometric multiplicity}\ 1 \\
& \text{and possesses a corresponding positive eigenvector}\ ,
\end{aligned}
\eqn
then it follows from Proposition \ref{B5}:

\begin{cor}\label{B11}
Suppose the assumptions of Theorem \ref{B10} together with \eqref{40}, \eqref{41}. Then, for each parameter value $n(\ve)$, $\ve\in (-\ve_0,\ve_0)\setminus\{0\}$ provided by Theorem \ref{B10} there exists a positive nontrivial equilibrium solution of the form $u(\ve)$ or $-u(\ve)$.
\end{cor}

\subsection{Example} Let $\Om\subset\R^N$, $N\ge 1$, be a bounded and smooth domain lying locally on one side of $\partial\Om$. Let $\partial\Om=\Gamma_0\cup\Gamma_1$, where $\Gamma_0,\Gamma_1$ are both open and closed in $\partial\Omega$ and $\Gamma_0\cap\Gamma_1=\emptyset$. Consider
$$
\Ac(u,x)w:=-\nabla_x\cdot\big(\mathpzc{a}(u,x)\nabla_x w\big)+\mathpzc{a}_1(u,x)\cdot\nabla_x w+\mathpzc{a}_0(u,x)w\ ,
$$
where
\bqn
\begin{aligned}\label{50}
&[u\mapsto \mathpzc{a}(u,\cdot)]\in C^m\big(\Phi,C^{1+\sigma}(\bar{\Om})\big)\ ,\\
& [u\mapsto \mathpzc{a}_1(u,\cdot)]\in C^m\big(\Phi,C^{\sigma}(\bar{\Om},\R^N)\big)\ ,\quad [u\mapsto \mathpzc{a}_0(u,\cdot)]\in C^m\big(\Phi,C^{\sigma}(\bar{\Om})\big)\ ,
\end{aligned}
\eqn
for some $m\ge 1$, $\sigma\in (0,1)$ small, and some open ball $\Phi$ in $C^{1+\sigma}(\bar{\Om})$ around $0$. Moreover, assume that
\bqn\label{51}
\mathpzc{a}(u,x)>0\ ,\quad x\in\bar{\Om}\ ,\quad u\in \Phi\ .
\eqn
Let 
\bqn\label{52}
\nu_0\in C^1(\Gamma_1)
\eqn
and let $\nu$ denote the outward unit normal to $\Gamma_1$. Let
\bqnn
\mathcal{B}(x)w:=\left\{\begin{array}{ll} w\ , & \text{on}\ \Gamma_0\ ,\\
 \frac{\partial}{\partial\nu}w+\nu_0(x) w\ , & \text{on}\ \Gamma_1\ .
\end{array}
\right.
\eqnn
Fix $p>N+2$ and let $E_0:=L_p(\Om)$ be ordered by its positive cone of functions that are nonnegative almost everywhere. Note that $E_0$ is a UMD-space. Set $E_1:=\Wqb^2(\Om)$, where
\bqnn
\Wqb^{2\xi}(\Om):=\left\{\begin{array}{ll} W_p^{2\xi}(\Om)\ ,\quad &0<{2\xi}<1/p\ ,\\
\big\{w\in W_p^{2\xi}(\Om)\,;\, u\vert_{\Gamma_0}=0\big\}\ ,& 1/p<{2\xi}<1+1/p\ ,\\
\big\{w\in W_p^{2\xi}(\Om)\,;\, \mathcal{B}u=0\big\}\ ,& {2\xi}>1+1/p\ ,
\end{array}
\right.
\eqnn
are subspaces of the usual Sobolev-Slobodeckii spaces $W_p^{2\xi}(\Om)$.
Note that if  $E_\xi:=(E_0,E_1)_{\xi,p}$, then
$$
E_{\varsigma}\,\dot{=}\,\Wqb^{2\varsigma}\dhr E_\alpha\,\dot{=}\,\Wqb^{2\alpha}\hookrightarrow C^{1+\sigma}(\bar{\Om})\ ,\quad 1+n/p+\sigma<2\alpha <2\varsigma\ ,
$$
where dots indicate equivalent norms \cite{Triebel}, and $\mathrm{int}(E_\varsigma^+)\ne\emptyset$.
Consider
$$
A(u)w:=\Ac(u,\cdot)w\ ,\quad w\in E_1\ ,\quad u\in \Phi\ .$$
Then \eqref{30} and \eqref{35} follow, e.g., from \cite{AmannIsrael}. Moreover, suppose that
\bqn\label{53}
\begin{aligned}
&\mathpzc{a}_0(0,x)\ge 0\ ,\quad \mathpzc{a}_1(0,x)=0\ ,\quad x\in\bar{\Om}\ ,\\
&\nu_0(x)\ge 0\ ,\quad x\in\Gamma_1\ ,
\end{aligned}
\eqn
and put $A_0:=A(0)$. According to \cite{AmannIsrael}, $-A_0$ is resolvent positive and generates a contraction semigroup on each $L_q(\Om)$, $1<q<\infty$, is self-adjoint in $L_2(\Om)$, and there exists a largest eigenvalue $\lambda_0\le 0$ of $-A_0\in\ml(E_1,E_0)$ with a positive eigenfunction $B\in E_\varsigma^+$. Moreover, \cite[Cor.13.6]{DanersKochMedina} ensures \eqref{t}. From \cite[III.Ex.4.7.3.d)]{LQPP} we deduce now \eqref{31} for each $\omega_0>0$. Given $a_m\in (0,\infty]$ and some open ball $\Sigma$ in $\E_1= L_p(J,\Wqb^2(\Om))\cap W_p^1(J,L_p(\Om))$ centered $0$ suppose that
\bqn\label{55}
\mu\ \text{satisfies}\ \eqref{32}\, ,\ \eqref{33}\ \text{with}\ \omega_0=0\, ,\ \text{and}\ \eqref{34}\ .
\eqn
Thus, if $b_0:=b(0,\cdot)$ for $b:=[u\mapsto b(u,\cdot)]\in C^m(\Sigma,L_{p'}^+(J))$ is nontrivial and normalized such that
\bqn\label{56}
\int_0^{a_m} b_0(a)\,e^{-\int_0^a\mu_0(r)\rd r}\,e^{\lambda_0 a}\,\rd a=1\ ,
\eqn
then $e^{-a A_0}B=e^{a\lambda_0 }B$ entails $Q_0B=B$, where
$$
Q_0=\int_0^{a_m} b_0(a)\,e^{-\int_0^a\mu_0(r)\rd r}\,e^{- a A_0}\,\rd a\in\mathcal{K}_+(E_{\varsigma})\ .
$$
Thus $r(Q_0)=1$ by \eqref{56} since $r(Q_0)$ is the only eigenvalue with positive eigenfunction. Therefore, Theorem~\ref{B10} entails:

\begin{prop}\label{vvvv}
Let $p>N+2$ and suppose \eqref{36}, \eqref{50}-\eqref{56}. Then the problem
\begin{align*}
&\partial_a u+\Ac\big(u(a),x\big)u+\mu(u,a)u=0\ ,\quad a\in J\ ,\quad x\in\Om\ ,\\
&u(0)=n\int_0^\infty b(u,a)u(a)\,\rd a\ ,\quad x\in\Om\ ,\\
&\mathcal{B}(x)u=0\ ,\quad a>0\ ,\quad x\in\partial\Om\ ,
\end{align*}
has a branch of nontrivial solutions $$\big(n(\ve),u(\ve)\big)\in\R^+\times \big(L_p(J,\Wqb^2(\Om))\cap W_p^1(J,L_p(\Om))\big)\ ,\quad 0<\vert\ve\vert<\ve_0\ ,$$ bifurcating from $(n,u)=(1,0)$, such that $u(\ve)(a)\in L_p^+(\Om)$ for $a\in J$ and $\ve>0$ small.
\end{prop}

\begin{rem}
The proposition above also holds if $E_0:=L_q(\Om)$ and $E_1:=W_{q,\mathcal{B}}^2(\Om)$ for $q>N+2$ different from $p\in (1,\infty)$. The only difference is that the interpolation space $E_\varsigma$ equals a subspace of the Besov space $B_{q,p}^{2\varsigma}(\Om)$.
\end{rem}

\subsection{Example}

We may also consider a functional dependence of $A$ on $u$. Indeed, let again $a_m\in (0,\infty]$ and let $\Om$, $E_1$, and $E_0$ be as in the previous example with $p\in (1,\infty)$ arbitrary. Given $u\in\E_1\hookrightarrow L_1(\R^+,L_p(\Om))$ let $U:=\int_0^\infty u(a)\rd a$ and consider $A(u)w:=\Ac(U,\cdot)w$ for $w\in E_1=\Wqb^2(\Om)$ with $\Ac$, $\mathcal{B}$ as in the previous example satisfying \eqref{51}-\eqref{53} but $\Phi$ in \eqref{50} is now an open ball in $L_p(\Om)$ centered at 0.
Suppose \eqref{55} with $\inf_{a>0}\mu(u,a)>\mathrm{type}(-A(u))$ for $u\in \Sigma$, \eqref{55}, \eqref{56} and $b\in C^m(\Sigma,L_{p'}^+(J))$. Moreover, assume that $b(u)\not\equiv 0$.
Then, analogously to the previous example,
$$
Q_u:=\int_0^{a_m} b(u,a)\, e^{-\int_0^a\mu(u,r)\rd r}\,e^{- aA(u)}\,\rd a\in\mathcal{K}_+(E_{\varsigma})
$$
is strongly positive for each $u\in\Sigma$, whence \eqref{266} by \cite[Thm.12.3, Cor.13.6]{DanersKochMedina}. We obtain from Theorem~\ref{B10} a branch of nontrivial solutions $$\big(n(\ve),u(\ve)\big)\in\R^+\times \big(L_p(J,\Wqb^2(\Om))\cap W_p^1(J,L_p(\Om))\big)\ ,\quad 0<\vert\ve\vert<\ve_0\ ,$$ to the problem
\begin{align*}
&\partial_a u+\Ac(U,x)u+\mu(u,a)u=0\ ,\quad a\in J\ ,\quad x\in\Om\ ,\\
&u(0)=n\int_0^{a_m} b(u,a)u(a)\,\rd a\ ,\quad x\in\Om\ ,\\
&\mathcal{B}(x)u=0\ ,\quad a>0\ ,\quad x\in\partial\Om\ ,
\end{align*}
bifurcating from $(n,u)=(1,0)$, such that $u(\ve)$ is positive for $\ve>0$ sufficiently small. If $\lambda_0(u)$ denotes the largest eigenvalue of $-A(u)\in\ml(E_1,E_0)$ for $u\in \Sigma$ and if
\bqn\label{zzz}
\int_0^{a_m} b(u,a)\,e^{-\int_0^a \mu(u,r)\rd r}\,e^{a\lambda_0(u)}\,\rd a\le 1\ ,\quad u\in \Sigma\ ,
\eqn
then
\bqnn
r(Q_u)\le 1\ ,\quad u\in\Sigma\ .
\eqnn
Indeed, if $B_u$ is a positive eigenfunction corresponding to $\lambda_0(u)$, then $e^{-aA(u)}B_u=e^{\lambda(u)a}B_u$ entails
$$
Q_u B_u=\int_0^{a_m} b(u,a)\,e^{-\int_0^a\mu(u,r)\rd r}\,e^{\lambda_0(u) a}\,\rd a\, B_u\ ,
$$
whence $r(Q_u)$ equals the left hand side of \eqref{zzz}. Recalling \eqref{28} we deduce that bifurcation must be supercritical provided \eqref{zzz} holds; that is, for $\ve\ge 0$ small we have $n(\ve)\ge 1$ and $u(\ve)$ is nonnegative. Note that $\lambda_0(u)\le 0$ if 
$\mathpzc{a}_0(u,\cdot)\ge 0$ and $\mathpzc{a}_0(u,\cdot)-\mathrm{div}(\mathpzc{a}_1(u,\cdot))\ge 0$ in $\Om$, 
$\nu_0\ge 0$ and $\mathpzc{a}_1(u,\cdot)\cdot\nu\ge 0$ on $\Gamma_1$
(see \cite[Rem.11.3]{AmannIsrael}) in which case the term $e^{\lambda_0(u)a}$ in \eqref{zzz} can be neglected. Moreover, $\mathrm{type}(-A(u))\le 0$ in this case. If the functions $\mathpzc{a}$, $\mathpzc{a}_1$, $\mathpzc{a}_0$ as well as $\mu$ and $b$ are symmetric with respect to $u$, that is, if $\mathpzc{a}(u,\cdot)=\mathpzc{a}(-u,\cdot)$ etc., then Proposition \ref{B5} entails that there is a positive equilibrium solution for any value of $n(\ve)$, $-\ve_0<\ve<\ve_0$.

\subsection{Example}

Let $a_m\in (0,\infty]$ and let $\Omega\subset\R^N$, $N\ge 1$, be a bounded and smooth domain. Given $u\in \E_1$ let $U:=\int_0^{a_m}u(a)\rd a$ and consider
$$
A(u)w:=-\nabla_x\cdot\big(\mathpzc{a}(U,x)\nabla_xw\big)\ ,\quad w\in E_1:=\Wqb^2(\Om):=\{v\in W_p^2(\Om)\,;\,\partial_\nu v=0\}
$$
so that $-A(u)\in\mathcal{H}(E_1,E_0)$ for $E_0:=L_p(\Om)$ and $p\in (1,\infty)$ provided that $\mathpzc{a}\in C^m(L_p(\Om),C^1(\bar{\Om}))$ is such that $\mathpzc{a}(U,x)>0$ for $x\in\bar{\Om}$. If $(n,u)$ is any positive solution to the problem
\begin{align*}
&\partial_a u+A(u)u+\mu(u,a)u=0\ ,\quad a>0\ ,\\
&u(0)=n\int_0^{a_m} b(u,a)u(a)\,\rd a\  ,
\end{align*}
then the relation $u(0)=nQ_u u(0)$ must hold, where 
$$
Q_u:=\int_0^{a_m} b(u,a)\, e^{-\int_0^a\mu(u,r)\rd r}\,e^{- a A(u)}\,\rd a\ .
$$
Therefore, owing to the fact that
\bqn\label{ttt}
\int_\Om e^{-a A(u)}\phi\,\rd x= \int_\Om \phi\,\rd x\ ,\quad \phi\in L_p(\Om)\ ,
\eqn
it follows by integrating the previous relation that necessarily
\bqn\label{zz}
1=n\int_0^{a_m}b(u,a)\, e^{-\int_0^a\mu(u,r)\rd r}\,\rd a=:n\,q(u)\ ,
\eqn
which is the same constraint as in the non-diffusive case (see \cite{CushingI}). Let $b\in C^m(\E_1,L_{p'}^+(J))$ and suppose \eqref{55}, \eqref{56} with $\lambda_0=0$ (in particular $q(0)=1$). Further assume that
\bqn\label{zu}
b(u,a)\le b(0,a)=b_0(a)\ ,\quad \mu(u,a)\ge \mu(0,a)=\mu_0(a)
\eqn
for $a\in J$ and $u\in\E_1^+$, which is a common modeling assumption stating that effects of population densities do neither increase fertility nor decrease mortality. Then $q(u(\ve))\le q(0)=1$ for the positive solution $(n(\ve),u(\ve))$, $\ve>0$ small, provided by Theorem \ref{B10}. Thus \eqref{zz} entails $n(\ve)\ge 1$ for $\ve>0$ small, that is, bifurcation must be supercritical, and there is no equilibrium solution other than the trivial $u\equiv 0$ corresponding to a parameter value $n<1$.

We shall point out that the present example simply reflects the non-diffusive case in the sense that our results here actually follow from the case $A\equiv0$ (see \cite{CushingI}). For this it is enough to observe that $\lambda_0=0$ is an eigenvalue of $-A(u)$ with corresponding constant eigenfunctions.

Moreover, taking $B={\bf 1}$ we have $\Pi_0(a,0)B=e^{-\int_0^a\mu_0(s)\rd s}$ and since the projection onto $M^\perp$ in Proposition \ref{A3} is given by
$$
1-P_M=\big[w\mapsto \frac{1}{\vert\Om\vert}\int_\Om w(x)\,\rd x\big]\ ,
$$
the direction of  bifurcation, given by $\zeta$ in Proposition \ref{A3}, can in principle be computed  explicitly using \eqref{ttt} also if one does not assume \eqref{zu}.

\begin{rem}
In all our examples we omitted a dependence of $\mu$ and $b$ on the spatial variable for simplicity. However, it is clear that such a dependence can be included as well.
\end{rem}


\begin{thebibliography}{99}


\bibitem{AmannIsrael}
H. Amann. \textit{Dual semigroups and second order linear elliptic boundary value problems.} Israel J. Math. {\bf 45} (1983), 225-254.

\bibitem{LQPP}
H. Amann. \textit{Linear and Quasilinear Parabolic Problems,
{Volume} {I}: Abstract Linear Theory.} Birkh\"auser, Basel,
Boston, Berlin 1995.


\bibitem{AmannEscherII}
H. Amann, J. Escher. {\it Analysis II.} Birkh\"auser, Basel 1999.


\bibitem{BusenbergIannelli3}
S. Busenberg, M. Iannelli. {\em A system of nonlinear degenerate parabolic equations.} J. Reine Angew. Math. {\bf 371} (1986), 1-15. 

\bibitem{Cushing0}
J.M. Cushing. {\it Nontrivial periodic solutions of integrodifferential equations.} J. Integral Equations
{\bf 1} (1979), 165-181.


\bibitem{CushingI}
J.M. Cushing. {\it Existence and stability of equilibria in age-structured
population dynamics.} J. Math. Biology {\bf 20} (1984), 259-276.

\bibitem{CushingII}
J.M. Cushing. {\it Global branches of equilibrium solutions of the McKendrick equations for age structured population growth.} Comp. Math. Appl. {\bf 11} (1985), 175-188.

\bibitem{CushingIII}
J. Cushing. {\it Equilibria in structured populations.} J. Math. Biology {\bf 23} (1985), 15-39.



\bibitem{DanersKochMedina}
D. Daners, P. Koch-Medina. {\it Abstract Evolution Equations, Periodic Problems, and Applications.}
Pitman Res. Notes Math. Ser., {\bf 279}, Longman, Harlow 1992. 

\bibitem{DelgadoEtAl}
M. Delgado, M. Molina-Becerra, A. Su\'arez. {\it A nonlinear age-dependent model with spatial diffusion.}
J. Math. Anal. Appl. {\bf 313} (2006), 366-380. 

\bibitem{DelgadoEtAl2}
M. Delgado, M. Molina-Becerra, A. Su\'arez. {\it Nonlinear age-dependent diffusive equations: A bifurcation approach.}
J. Diff. Equations {\bf 244} (2008), 2133-2155. 


\bibitem{FragnelliManiar}
G. Fragnelli, A. Idrissi, L. Maniar. \textit{The asymptotic behavior of a population equation with diffusion and delayed birth process.} Discrete Contin. Dyn. Syst. Ser. B {\bf 7} (2007), 735-754.


\bibitem{Iannelli}
M. Iannelli, M. Martcheva, F.A. Milner. {\it Gender structured population modeling.} SIAM Frontiers in Applied Mathematics. Philadelphia 2005.


\bibitem{LanglaisJMB}
M. Langlais. {\em Large time behavior in a nonlinear age-dependent population dynamics problem with spatial diffusion.} J. Math. Biol. {\bf 26} (1988), 319-346.

\bibitem{LanglaisSIAM}
M. Langlais. {\em A nonlinear problem in age-dependent population diffusion.} SIAM J. Math. Anal. {\bf 16} (1985), 510-529.

\bibitem{LanglaisBusenberg}
M. Langlais, S. Busenberg. {\em Global behaviour in age structured S.I.S. models with seasonal periodicities and vertical transmission.}  J. Math. Anal. Appl. {\bf 213} (1997), 511-533. 

\bibitem{LauWalJMPA}
Ph. Lauren\c{c}ot, Ch. Walker. {\em Proteus mirabilis swarm-colony development with drift.} To appear in: J. Math. Pures Appl.

\bibitem{MagalThieme}
P. Magal, H. Thieme. {\em Eventual compactness for semiflows generated by nonlinear age-structured models.} Commun. Pure Appl. Anal. {\bf 3}, No. 4 (2004), 695-727.

\bibitem{Pruess1}
J. Pr\"u\ss . {\it Equilibrium solutions of age-specific population dynamics of several species.} J. Math.
Biol. {\bf 11} (1981), 65-84.

\bibitem{Pruess2}
J. Pr\"u\ss . {\it On the qualitative behaviour of populations with age-specific interactions.} Comput.
Math. Appl. {\bf 9} (1983), 327-339.

\bibitem{Pruess3}
J. Pr\"u\ss . {\it Stability analysis for equilibria in age-specific population dynamics.} Nonl. Anal. Th.
Math. Appl. {\bf 7} (1983), 1291-1313.

\bibitem{RhandiSchnaubelt}
A. Rhandi, R. Schnaubelt. {\em Asymptotic behaviour of a
non-autonomous population equation with diffusion in $L^1$.} Discrete Contin. Dyn. Syst. {\bf 5} (1999), 663-683.

\bibitem{Rudin}
W. Rudin. {\it Functional Analysis.} International Series in Pure and Applied Mathematics. $2^{nd}$ edition.
McGraw-Hill Inc. 1991.

\bibitem{Saal}
J. Saal. {\em Maximal regularity for the Stokes system on noncylindrical space-time domains.} J. Math. Soc. Japan {\bf 58} (2006), 617-641.


\bibitem{Schaefer}
H.H. Schaefer. {\it Topological Vector Spaces.} Graduate Texts in Mathematics. Springer Verlag, New York, Heidelberg, Berlin 1971.



\bibitem{ThiemeDCDS98}
H. Thieme. {\em Positive perturbation of operator semigroups: growth bounds, essential compactness, and asynchronous exponential growth.} Discrete Contin. Dyn. Syst. {\bf 4} (1998), 735-764.

\bibitem {Triebel}
H. Triebel. \textit{Interpolation theory, function spaces,
differential operators.} Second edition. Johann Ambrosius Barth. Heidelberg, Leipzig 1995.


\bibitem{WalkerEJAM}
Ch. Walker. {\em Global existence for an age and spatially structured haptotaxis model with nonlinear age-boundary conditions.}
Europ. J. Appl. Math. {\bf 19} (2008), 113-147. 

\bibitem{WalkerDCDS}
Ch. Walker. {\em Age-Dependent Equations with Non-Linear Diffusion.}
Preprint (2008). 

\bibitem{Webb}
G.F. Webb. {\em Theory of nonlinear age-dependent population dynamics.} Marcel Dekker, New York 1985.

\bibitem{WebbSpringer}
G.F. Webb. {\em Population models structured by age, size, and spatial position.} In: P. Magal, S. Ruan (eds.) {\em Structured Population Models in Biology and Epidemiology.} Lecture Notes in Mathematics, Vol. 1936. Springer, Berlin, 2008.


\end{thebibliography}
\end{document}